\documentclass[12pt,reqno]{amsart}%
\usepackage{eurosym}
\usepackage{amssymb}
\usepackage{amsfonts}
\usepackage{amsmath}
\usepackage{amsthm}
\usepackage[hidelinks]{hyperref}
\usepackage{times}
\usepackage{latexsym}
\usepackage[mathscr]{eucal}
\usepackage{graphicx}%
\providecommand{\U}[1]{\protect\rule{.1in}{.1in}}
\def\equationautorefname~#1\null{(#1)\null}

\theoremstyle{plain}

\newtheorem{algorithm}{Algorithm}[section]
\newtheorem{thm}{Thm}

\newtheorem{corollary}[algorithm]{Corollary}

\newtheorem{definition}[algorithm]{Definition}
\newtheorem*{Remark-delta}{Remark on all things $\protect\delta$}

\newtheorem{lemma}[algorithm]{Lemma}

\newtheorem{theorem} [algorithm] {Theorem}
\newtheorem{theoremlet}[thm]{Theorem}

\newtheorem{lemmalet}[thm]{Lemma}

\newtheorem{proposition}[algorithm]{Proposition}

\theoremstyle{definition}

\newtheorem{example}[algorithm]{Example}
\newtheorem{remark}[algorithm]{Remark}
\newtheorem*{remarknonum}{Remark}

\renewcommand{\Lambda}{\mathcal{L}}
\DeclareMathOperator{\spann}{span}
\DeclareMathOperator{\Trace}{Trace}
\DeclareMathOperator{\Tr}{Trace}
\DeclareMathOperator{\Ric}{Ric}
\DeclareMathOperator{\inx}{index}
\DeclareMathOperator{\Index}{index}

\DeclareMathOperator{\dist}{dist}

\DeclareMathOperator{\foc}{foc}
\DeclareMathOperator{\length}{length}
\DeclareMathOperator{\sect}{sec}
\begin{document}
\title{A Softer Connectivity Principle}
\author{Luis Guijarro}
\address{Department of Mathematics, Universidad Aut\'{o}noma de Madrid, and ICMAT
CSIC-UAM-UCM-UC3M }
\email{luis.guijarro@uam.es}
\urladdr{http://verso.mat.uam.es/~luis.guijarro}
\author{Frederick Wilhelm}
\address{Department of Mathematics\\
University of California\\
Riverside, CA 92521}
\email{fred@math.ucr.edu}
\urladdr{https://sites.google.com/site/frederickhwilhelmjr/home}
\thanks{The first author was supported by research grants MTM2014-57769-3-P and
MTM2017-85934-C3-2-P from the MINECO, and by ICMAT Severo Ochoa project
SEV-2015-0554 (MINECO). }
\thanks{This work was supported by a grant from the Simons Foundation (\#358068,
Frederick Wilhelm)}
\date{\today}

\begin{abstract}
We give soft, quantitatively optimal extensions of the classical Sphere
Theorem, Wilking's connectivity principle and Frankel's Theorem to the context
of $\Ric_{k}$ curvature. The hypotheses are soft in the sense that they are
satisfied on sets of metrics that are open in the $C^{2}$--topology.

\end{abstract}
\maketitle

\pdfbookmark[1]{Introduction}{Introduction}

A Riemannian manifold $M$ has $k^{th}$--intermediate Ricci curvature $\geq
\ell$ if for any orthonormal $\left(  k+1\right)  $--frame $\left\{
v,w_{1},w_{2},\ldots,w_{k}\right\}  ,$ the sectional curvature sum,
$\Sigma_{i=1}^{k}\mathrm{sec}\left(  v,w_{i}\right)  ,$ is $\geq\ell$
(\cite{Wu}, \cite{Shen}). For brevity we write $Ric_{k}\,M\geq\ell.$ In this
article, we consider some of the topological implications that positive
$\Ric_{k}$ curvature imposes on a manifold.

We start by combining Berger's proof of the diameter sphere theorem,
\cite{Pet}, with the Jacobi field comparison lemma from \cite{GuijWilh} to
obtain the following generalization of the classical $\frac{1}{4}$--pinched
sphere theorem.


\begin{theoremlet}
\label{Intermeadiate Ricci Thm}Let $M$ be a complete Riemannian $n$--manifold
with $Ric_{k}\geq k.$ If there is a point $p\in M$ with $\mathrm{conj}%
_{p}>\frac{\pi}{2},$ then the universal cover of $M$ is $\left(  n-k\right)
$--connected. In particular, if $k\leq\frac{n}{2},$ then the universal cover
of $M$ is homeomorphic to the sphere.
\end{theoremlet}

Since manifolds with $1\leq\sect <4$ have $\Ric_{1}\geq1$ and conjugate radii
$>\frac{\pi}{2},$ Theorem \ref{Intermeadiate Ricci Thm} generalizes the
classical sphere theorem.


The assumption $\mathrm{conj}_{p}>\frac{\pi}{2}$ cannot be weakened. Indeed,
$\mathrm{Ric}_{n-2}\left(  \mathbb{C}P^{n}\right)  \geq n-2,$ while
$\mathrm{conj}\left(  \mathbb{C}P^{n}\right)  =\frac{\pi}{2},$ yet
$\mathbb{C}P^{n}$ is not $2$--connected. In Section \ref{Sect: Examples}, we
give examples of various versions of Theorem \ref{Intermeadiate Ricci Thm},
including examples that show the connectivity estimate cannot be improved.

For our second result, let $N$ be a smoothly embedded submanifold of a
Riemannian manifold $M$, and denote by $S_{v}:T_{p}N\to T_{p}N$ the shape
operator of $N$ corresponding to a unit vector $v$ normal to $N$.
For brevity we write $S_{v}|_{W}$ for the composition of $S_{v}$ restricted to
some subspace $W$ of $T_{p}N$ with orthogonal projection $T_{p}%
N\longrightarrow W$. We also denote by $\foc_{N}$ the focal radius of $N$.

\begin{theoremlet}
\label{quant conn with norm II thm} Let $M$ be a simply connected, complete
Riemannian $n$--manifold with $\mathrm{Ric}_{k}M\geq k,$ and let $N\subset M$
be a compact, connected, embedded, $\ell$--dimensional submanifold.

If for some $r\in\left[  0,\frac{\pi}{2}\right)  ,$%
\begin{equation}
\foc_{N}>r, \label{str focalinequal}%
\end{equation}
and for all unit vectors $v$ normal to $N$ and all $k$--dimensional subspaces
$W\subset$ $TN$,%
\begin{equation}
\left\vert \Tr\left(  S_{v}|_{W}\right)  \right\vert \leq k\cot\left(
\frac{\pi}{2}-r\right)  , \label{II hypo}%
\end{equation}
then the inclusion $N\hookrightarrow M$ is $\left(  2\ell-n-k+2\right)  $--connected.
\end{theoremlet}

Recall that an inclusion $\iota:N\hookrightarrow M$ is called $q$--connected
if the relative homotopy groups $\pi_{k}(M,N)=0$ for all $k\leq q$.
Equivalently, $\iota$ induces an isomorphism on homotopy groups up to
dimension $\left(  q-1\right)  $ and a surjection of $q^{th}$--homotopy groups.

In the event that $\pi_{1}\left(  M\right)  \neq0$, Theorem
\ref{quant conn with norm II thm} implies $\pi_{k}(M,N)=0$ for all integers
$k$ in the range $2\leq k\leq2\ell-n-k+2$. To see this, apply Theorem
\ref{quant conn with norm II thm} the the universal cover of $M$ and use the
fact that covering maps induce isomorphisms on higher relative homotopy groups
(Exercise 6, p. 358, \cite{Hatch}).

The Clifford torus in $\mathbb{S}^{3}$ shows that Theorem
\ref{quant conn with norm II thm} is optimal in the sense that the conclusion
is false if (\ref{II hypo}) is replaced with a weak inequality (see Example
\ref{Ex: cliff torus}).

The proofs of both theorems also work in the non-simply connected case if the
conjugate and focal radius hypotheses are replaced with the corresponding
hypotheses on injectivity and normal injectivity radius. If, in addition, we
assume that $r=0$ and $k=1,$ then Theorem \ref{quant conn with norm II thm}
recovers Wilking's connectivity principle (Theorem 1, \cite{Wilk3}). Theorem
\ref{quant conn with norm II thm} also generalizes the various versions of the
connectivity principle in \cite{FangRong}. In the case $r=0$ and $k=n-1,$ the
proof of Theorem \ref{quant conn with norm II thm} recovers Frankel's theorem
that the inclusion of a minimal hypersurface in a manifold with positive Ricci
curvature induces a surjection of fundamental groups (\cite{Frank2}).
Similarly, the proof of Theorem \ref{Intermeadiate Ricci Thm} recovers
Petersen's result that a manifold with $\mathrm{Ric}\geq n-1$ is simply
connected if it has a point whose injectivity radius is $>\frac{\pi}{2}$
(\cite{Pet}, page 181).

We also include a generalization of Frankel's Theorem on intersections of
totally geodesic submanifolds in positive sectional curvature. We allow at the
same time for positive $\Ric_{k}$ curvature and non-totally geodesic
submanifolds. The result is quantitatively optimal and appears as Theorem
\ref{quant Frankel}.

The proofs of Theorems \ref{Intermeadiate Ricci Thm} and
\ref{quant conn with norm II thm} use a version of the Long Homotopy Lemma
that applies to submanifolds. We have not been able to find the required
formulation in the literature, so we include it here (see, e.g., page 235 in
\cite{doCarm} for the Long Homotopy Lemma for points. Also see \cite{AbrMey}
or \cite{CheegGrm}).

\begin{lemmalet}
\label{Lem: long homot}(Long Homotopy Lemma for Submanifolds) Let $M$ be a
Riemannian manifold and $N$ a closed, embedded submanifold. Let $\gamma
:\left[  0,b\right]  \longrightarrow M$ be a unit speed geodesic with
$b<2\foc_{N}$ and $\gamma\left(  0\right)  ,\gamma\left(  b\right)  \in N$
with $\gamma^{\prime}\left(  0\right)  ,\gamma^{\prime}\left(  b\right)  \perp
N$. Suppose that $H$ is a homotopy so that
\begin{align*}
H_{1}\left(  t\right)   &  =\gamma\left(  t\right)  ,\\
H_{s}(0),H_{s}(b)  &  \in N\quad\text{ for all }s\in\left[  0,1\right]
,\text{ and}\\
H_{0}\left(  t\right)   &  \in N\quad\text{ for all }t\in\left[  0,b\right]  .
\end{align*}
Then for some $t$, the length of the curve $H_{t}\geq2\foc_{N}$.
\end{lemmalet}

The paper is structured as follows: We start with a section reviewing the
comparison theorem for Jacobi fields that we need in the rest of the paper and
use it to derive Corollary \ref{first cor} about the index of Lagrangians in
$k$-th Ricci curvature. Section \ref{sect:index} recalls the definition of the
index of a geodesic both for the endpoint and the endmanifold cases. The main
result in this section is Theorem \ref{hingst-kalish thmm}, which is due to N.
Hingston and D. Kalish. The proof of Theorem \ref{Intermeadiate Ricci Thm}
appears in Section \ref{sect:proof_Thm_A}, and the proof of the quantitative
connectivity principle (Theorem \ref{quant conn with norm II thm}) is in
Section \ref{sect:proof_of_thm_B}. Section \ref{sect:Frankel} contains the
proof of our extension of Frankel's Theorem, and Section \ref{Sect: Examples}
gives some examples showing the sharpness of our theorems. The proof of the
Long Homotopy Lemma for submanifolds has been postponed to an appendix.

\begin{remarknonum}
An advantage of Theorem \ref{quant conn with norm II thm} and Theorem
\ref{quant Frankel} is that they apply to an open set of metrics. The main
theorem of \cite{MurWilh} shows that the set of metrics to which either the
original connectivity principle or Frankel's theorem can be applied is residual.
\end{remarknonum}



\begin{remarknonum}
Our techniques yield versions of Theorems \ref{Intermeadiate Ricci Thm},
\ref{quant conn with norm II thm}, and \ref{quant Frankel} for radial
intermediate Ricci curvature. We leave the details to the reader.
\end{remarknonum}

\noindent\textbf{Acknowledgment: } We thank Burkhard Wilking for reminding us
about the long homotopy lemma after seeing a preliminary version of some of
the results in this paper. We are grateful to Paula Bergen for copyediting the
manuscript.

\medskip

\section{Transverse Jacobi Field Comparison\label{Jacobi-comp sect}}

In this section, we review Wilking's transverse Jacobi equation and some
comparison results developed for it in \cite{GuijWilh}. To simplify the
writing, any vectors or vector fields along a geodesic will always be normal
to it, and we assume throughout that $\dim M=n$.

\subsection{Lagrangian subspaces of Jacobi fields}

Let $\gamma:I\rightarrow M$ be a unit speed geodesic, and denote by
$\mathcal{J}$ the vector space of normal Jacobi fields along it. $\mathcal{J}$
is a vector space of dimension $2n-2$. Using symmetries of the curvature
tensor, we see that for any pair $J_{1},J_{2}\in\mathcal{J},$
\[
\omega\left(  J_{1},J_{2}\right)  =\left\langle J_{1}^{\prime},J_{2}%
\right\rangle -\left\langle J_{1},J_{2}^{\prime}\right\rangle ,
\]
is constant along $\gamma$ and hence defines a symplectic form $\omega$ on
$\mathcal{J}$.

\begin{definition}
We will say that an $(n-1)$--dimensional subspace $\Lambda$ of $\mathcal{J} $
is Lagrangian when $\omega$ vanishes on $\Lambda$.
\end{definition}



\begin{definition}
Let $\mathcal{V}$ be a subspace of $\Lambda$. We will say that $\mathcal{V}$
has full index at $\bar{t}$ if any $J\in\Lambda$ with $J(\bar{t})=0\/$ belongs
to $\mathcal{V}$. We will also say that $\mathcal{V}$ has full index on an
interval $I$ if it has full index at each point of $I$.
\end{definition}

For a fixed interval $I,$ we denote by $\mathcal{K}$ the minimal subspace of
$\Lambda$ that has full index on $I$. Thus
\begin{equation}
\label{eq:index_subspace}\mathcal{K}\equiv\spann\left\{  \left.  J\in
\Lambda\text{ }\right\vert \text{ }J\left(  t\right)  =0\text{ for some }t\in
I\right\}  .
\end{equation}

\subsection{The Riccati operator related to a Lagrangian}

The set of times $t$ where
\begin{equation}
\left\{  J(t)\ |\ J\in\Lambda\right\}  =\dot{\gamma}\left(  t\right)  ^{\perp}
\label{Eval 1-1 eqn}%
\end{equation}
is open and dense in $I$. Observe that these are the times where no nontrivial
field in $\Lambda$ vanishes. For these $t$ we get a well-defined Riccati
operator
\begin{equation}
S_{t}:\dot{\gamma}\left(  t\right)  ^{\perp}\longrightarrow\dot{\gamma}\left(
t\right)  ^{\perp},\qquad S_{t}(v)=J_{v}^{\prime}\left(  t\right)  ,
\label{Rica dfn}%
\end{equation}
where $J_{v}$ is the unique $J_{v}\in\Lambda$ so that $J_{v}\left(  t\right)
=v.$ The Jacobi equation then decomposes into the two first order equations
\[
S_{t}\left(  J\right)  =J^{\prime},\qquad S_{t}^{\prime}+S_{t}^{2}+R=0,
\]
where $S_{t}^{\prime}$ is the covariant derivative of $S_{t}$ along $\gamma$,
and $R$ is the curvature along $\gamma$, that is $R\left(  \cdot\right)
=R\left(  \cdot,\dot{\gamma}\right)  \dot{\gamma}$ (see Equation 1.7.1 in
\cite{gw}). Although $S_{t}$ depends on the choice of $\Lambda$, we will never
discuss more than one family $\Lambda$, so we omit this dependence from our notation.

Given $t\in I$, and a subspace $\mathcal{V}\subset\Lambda$, we get a subspace
of $\gamma^{\prime\perp}$ by setting
\begin{equation}
\mathcal{V}(t)=\left\{  \left.  J\left(  t\right)  \text{ }\right\vert \text{
}J\in\mathcal{V}\right\}  \oplus\left\{  \left.  J^{\prime}\left(  t\right)
\text{ }\right\vert \text{ }J\in\mathcal{V}\text{ and }J\left(  t\right)
=0\right\}  . \label{dfn of eval eqn}%
\end{equation}
The second summand vanishes for almost every $t$ (cf Lemma 1.7 of \cite{gw}),
even at such times $S_{t}$ is well defined on a subspace of $\gamma
^{\prime\perp}.$ Indeed,

\begin{lemma}
\label{S well dfned}(\cite[Remark 3.3]{GuijWilh}) Let $\mathcal{K}$ the
subspace of $\Lambda$ defined in \eqref{eq:index_subspace}. Then equation
\eqref{Rica dfn} defines $S_{t}$ on $\mathcal{K}\left(  t\right)  ^{\perp}$
for every $t\in I$.
\end{lemma}

When $\Lambda$ has fields vanishing at $t\in I$ and $v\in\mathcal{K}%
(t)^{\perp}$, there are many $J_{v}$ in $\Lambda$ with $J_{v}(t)=v$; Lemma
\ref{S well dfned} says that $J_{v}^{\prime}\left(  t\right)  $ is independent
of these choices.

\begin{example}
\label{ex:submanifold_Lagrangian} An especially useful Lagrangian appears when
we take a geodesic $\gamma:\mathbb{R}\rightarrow M$ normal to a submanifold
$N$ at time zero and consider those Jacobi fields obtained from geodesic
variations that leave $N$ orthogonally. More precisely, let $p=\gamma(0)$, and
denote by $\mathrm{S}_{\gamma^{\prime}\left(  0\right)  }:T_{p}N\rightarrow
T_{p}N$ the shape operator of $N$ determined by $\gamma^{\prime}\left(
0\right)  $. The Lagrangian $\Lambda_{N}$ is the set of Jacobi fields along
$\gamma$ such that
\begin{equation}
J(0)\in T_{p}N\hspace{0.14in}\text{and}\hspace{0.14in}J^{\prime T}%
=S_{\gamma^{\prime}(0)}J(0), \label{dfn of Lambda_N}%
\end{equation}
where $J^{\prime T}$ is the component of $J^{\prime}(0)$ tangent to $S$
(details can be found at \cite[page 227]{doCarm}).

For this Lagrangian, the Riccati operator $S_{0}$ agrees with $S_{\gamma
^{\prime}(0)}$. This follows immediately from equation \eqref{dfn
of Lambda_N} above.
\end{example}

\subsection{Comparison Lemma for the Transverse Jacobi Equation}

\leavevmode

For a subspace $W_{t}\subset\gamma^{\prime}\left(  t\right)  ^{\perp}$ we will
use
\[
P_{W,t}:\gamma^{\prime\perp}\longrightarrow W_{t}%
\]
to denote the orthogonal projection onto $W_{t}$. For simplicity of notation
we will write
\[
\Trace S_{t}|_{W}\text{ for }\mathrm{Trace}\left(  P_{W,t}\circ S_{t}%
|_{W}\right)  .
\]

In the case of positive curvature, Lemma 4.23 in \cite{GuijWilh} gives us the following.

\begin{lemma}
[Intermediate Ricci Comparison]\label{sing soon Ric_k Lemma} Let
$\gamma:\left[  0,t_{\mathrm{max}}\right]  \longrightarrow M$ be a unit speed
geodesic in a complete Riemannian $n$--manifold $M$ with $\mathrm{Ric}%
_{k}\left(  \dot{\gamma},\cdot\right)  $ $\geq k.$ Let $\Lambda$ be a
Lagrangian subspace of normal Jacobi fields along $\gamma$ with Riccati
operator $S$, and let $W_{0}\perp\gamma^{\prime}(0)$ be a $k$--dimensional
subspace so that
\begin{equation}
\Trace S_{0}|_{W_{0}}\leq k\cdot\cot\left(  s_{0}\right)  \label{tr hyp}%
\end{equation}
for some $s_{0}\in\left(  0,\pi\right)  .$ Denote by $\mathcal{V}$ the
subspace of $\Lambda$ formed by those Jacobi fields that are orthogonal to
$W_{0}$ at $0,$ and by $H(t)$ the subspace of $\gamma^{\prime}\left(
t\right)  ^{\perp}$ that is orthogonal to $\mathcal{V}(t)$ at each
$t\in\left[  0,t_{\mathrm{max}}\right]  $
. Assume that $\mathcal{V}$ is of full index in the interval $\left[
0,t_{\mathrm{max}}\right]  $.

Then for all $t\in\left[  0,t_{\mathrm{max}}\right]  $,
%
\begin{equation}
\Tr S_{t}|_{H(t)}\leq k\cdot\cot\left(  t+s_{0}\right)  .
\label{small future Inequal}%
\end{equation}

\end{lemma}

By Remark 4.20 of \cite{GuijWilh} we also have

\begin{lemma}
\label{comp with sing at 0} Let $\gamma:\left[  0,t_{\mathrm{max}}\right]
\longrightarrow M$ be a unit speed geodesic in a complete Riemannian
$n$--manifold $M$ with $\mathrm{Ric}_{k}\left(  \dot{\gamma},\cdot\right)  $
$\geq k.$ Let $\Lambda$ be a Lagrangian subspace of normal Jacobi fields along
$\gamma$ with Riccati operator $S$. Let $\mathcal{V}\subset\Lambda$ be a
subspace of full index, and for $t\in\left[  0,t_{\mathrm{max}}\right]  ,$ let
$H(t)$ be a $k$--dimensional subspace of $\gamma^{\prime}\left(  t\right)
^{\perp}$ that is orthogonal to $\mathcal{V}(t).$

Then
\[
\Tr S_{t}|_{H(t)}\leq k\cdot\cot\left(  t\right)  .
\]
In particular, if
\begin{equation}
\dim\mathcal{V}\leq n-1-k,\text{ then }t_{\mathrm{max}}<\pi. \label{small
index implies short}%
\end{equation}

\end{lemma}

\begin{remark}
Like Lemma \ref{sing soon Ric_k Lemma}, Lemma \ref{comp with sing at 0},
follows by combining classical Riccati comparison (\cite{Esch},\cite{EschHein}%
) with the transverse the Jacobi equation from \cite{Wilk1}.

The analog in Lemma \ref{comp with sing at 0} of the initial trace Hypothesis
(\ref{tr hyp}) of Lemma \ref{sing soon Ric_k Lemma} is
\begin{equation}
\lim_{t\rightarrow0}\inf\left(  k\cot\left(  t\right)  -\Tr S_{t}%
|_{H(t)}\right)  \geq0. \label{nothing about everything}%
\end{equation}
Since (\ref{nothing about everything}) holds for all $k$--dimensional
subspaces of Jacobi fields along all geodesics in all Riemannian manifolds, it
is not included in the statement of Lemma \ref{comp with sing at 0}.
\end{remark}

There is an improved version of (\ref{small index implies short}) when
$\Lambda$ is nonsingular at $t=0.$ We give this in the next corollary which is
related to \cite[Proposition 3.4]{GonGui}, \cite[Theorem 1.1]{GuijWilh}, and
\cite[Theorem G]{GuamWilh}.

\begin{corollary}
\label{first cor} For $r>0,$ let $\gamma:\left[  0,\frac{\pi}{2}+r\right]
\longrightarrow M$ be a unit speed geodesic in a complete Riemannian
$n$--manifold $M$ with $Ric_{k}\left(  \dot{\gamma},\cdot\right)  $ $\geq k.$
Let $\Lambda$ be a Lagrangian subspace of normal Jacobi fields along $\gamma$
with Riccati operator $S.$ Set
\[
\mathcal{K}\equiv\left\{  \left.  J\in\Lambda\text{ }\right\vert \text{
}J\left(  t\right)  =0\text{ for some }t\in\left(  0,\frac{\pi}{2}+r\right]
\right\}  .
\]

If there is an $\ell$--dimensional subspace $U_{0}\subset\gamma^{\prime\perp
}$
so that for all $k$--dimensional subspaces $W\subset U_{0}$,
\begin{equation}
\label{init trace inequal}\Trace S_{0} |_{W} \leq k\cot\left(  \frac{\pi}%
{2}-r\right)  ,
\end{equation}
then $\dim\mathcal{K} \geq\ell-k+1.$
\end{corollary}

\begin{proof}
Suppose that $\dim\mathcal{K} <\ell-k+1.$ Since all of the following are
subspaces of the $\left(  n-1\right)  $--dimensional space $\gamma^{\prime
}\left(  0\right)  ^{\perp},$ we get that
\begin{align*}
\dim(U_{0}\cap\mathcal{K}\left(  0\right)  ^{\perp})  &  = \dim U_{0}
+\dim\mathcal{K}\left(  0\right)  ^{\perp}-\dim( U_{0}+\mathcal{K}\left(
0\right)  ^{\perp} )\\
&  \geq\dim U_{0} +n-1-\dim\mathcal{K}(0) -\left(  n-1\right) \\
&  \geq\dim U_{0} -\dim\mathcal{K}\\
&  >\ell-\left(  \ell-k+1\right) \\
&  =k-1.
\end{align*}

Since $\dim(U_{0}\cap\mathcal{K}\left(  0\right)  ^{\perp})\geq k,$ inequality
\eqref{init trace inequal} together with Lemma \ref{sing soon Ric_k Lemma}
gives us that for all values of $t\in\left(  0,\frac{\pi}{2}+r\right]  $,
there is a nonzero $v\in\mathcal{K}\left(  t\right)  ^{\perp}$ so that
\[
\left\langle S_{t}v,v\right\rangle \leq\cot\left(  \frac{\pi}{2}-r+t\right)
\left\vert v\right\vert ^{2}.
\]
But $\lim\nolimits_{s\rightarrow\pi^{-}}\cot\left(  s\right)  =-\infty$; thus
there is some $t_{1}\in\left(  0,\frac{\pi}{2}+r\right]  $ so that $S_{t}$ is
not well defined on $\mathcal{K}\left(  t\right)  ^{\perp}$. Since this
contradicts Lemma \ref{S well dfned}, $\dim\mathcal{K}\geq\ell-k+1$.
\end{proof}

\section{The index of a geodesic}

\label{sect:index} Some of the results of this section are well known and can
be found in \cite{Miln}; the exception is probably the Hingston-Kalish
Theorem, which can be found in \cite{HingKa}.

\subsection{The endpoints case}

For a pair of points $p$, $q\in M$, let $\Omega_{p,q}$ denote the space of
piecewise smooth paths connecting $p$ to $q,$ that are parameterized on
$\left[  0,1\right]  ,$ and let
\[
E:\Omega_{p,q}\rightarrow\lbrack0,\infty),\qquad E(\alpha)=\frac{1}{2}\int%
_{0}^{1}\,|\alpha^{\prime}|^{2}\,dt
\]
denote the energy function. $E$ is related to arc length via
\begin{equation}
\length\left(  \alpha\right)  \leq\sqrt{2E\left(  \alpha\right)  },
\label{Ener vs arc len}%
\end{equation}
with equality if and only if $\alpha$ has constant speed (\cite{Pet}, pages 182--183).

Following \cite{Miln}, we will work with a finite dimensional submanifold
$P\subset\Omega_{p,q}$ that is homotopy equivalent to $\Omega_{p,q}$, and we
will continue using $E$ for the restriction of the energy function to $P$. The
critical points of $E$ are the unit geodesics connecting $p$ to $q$.

\begin{definition}
Let $\gamma:[0,1]\rightarrow M$ be a geodesic between $p$ and $q$. The index
of $\gamma$ is the index of $\gamma$ as a critical point of $E$.
\end{definition}

When $p$ and $q$ are not conjugate along $\gamma$, $\gamma$ is a nondegenerate
critical point of $E$. When they are conjugate, a kernel appears agreeing with
the set of Jacobi fields vanishing at $p$ and $q$ simultaneously. In both
cases, the Morse Index Lemma affirms that the index of $\gamma$ agrees with
the number of conjugate points to $\gamma(0)$ along $\gamma$ in the interval
$(0,1)$.

\subsection{The endmanifold case}

Given two smooth embedded submanifolds $N$ and $\tilde{N}$ of a Riemannian
manifold $M$, we let $\Omega(N,\tilde{N})$ be the space of all piecewise
smooth curves from $N$ to $\tilde{N}$ parameterized on $\left[  0,b\right]  .$
It is well known that the critical points of the energy functional
$E:\Omega(N,\tilde{N})\longrightarrow\mathbb{R}$ are precisely the geodesics
$\gamma:\left[  0,b\right]  \longrightarrow M$ from $N$ to $\tilde{N}$ that
are normal to $N$ and $\tilde{N}$ at both endpoints.

In what follows, we only consider continuous variations $f:\left[  0,b\right]
\times\left(  -\varepsilon,\varepsilon\right)  \longrightarrow M$ of such
geodesics $\gamma$ with
\begin{equation}
f\left(  0,s\right)  \in N,\quad f\left(  b,s\right)  \in\tilde{N},\quad\text{
and }\quad V(t):=\frac{\partial f}{\partial s}(t,0). \label{ends in N
and  N tilde}%
\end{equation}
In addition, assume that there is a partition $0=t_{0}<t_{1}<t_{2}%
<\cdots<t_{k+1}=b$ so that
\begin{equation}
f|_{\left[  t_{i-1},t_{i}\right]  \times\left(  -\varepsilon,\varepsilon
\right)  \text{ }}\text{is smooth for all }i\in\left\{  1,2,\ldots,k\right\}
. \label{pw sm cond}%
\end{equation}
The second variation formula gives us
\[
E^{\prime\prime}\left(  0\right)  =-\int_{0}^{b}\left\langle V^{\prime\prime
}+R\left(  V,\dot{\gamma}\right)  \dot{\gamma},V\right\rangle dt+\left.
\left\langle V^{\prime}-S_{\gamma^{\prime}}V,V\right\rangle \right\vert
_{0}^{b}+\sum_{i}\left\langle V^{\prime}\left(  t_{i}^{-}\right)  -V^{\prime
}\left(  t_{i}^{+}\right)  ,V\left(  t_{i}\right)  \right\rangle ,
\]
where $S_{\gamma^{\prime}}$ is the shape operator of $N$ or $\tilde{N}$ for
the normal $\gamma^{\prime},$
\[
V^{\prime}\left(  t_{i}^{-}\right)  =\lim\nolimits_{t\rightarrow t_{i}^{-}%
}V^{\prime}\left(  t\right)  ,\quad\text{ and }\quad V^{\prime}\left(
t_{i}^{+}\right)  =\lim\nolimits_{t\rightarrow t_{i}^{+}}V^{\prime}\left(
t\right)  .
\]
(See, for instance, \cite[pages 198-199]{doCarm})

As usual the index of the critical point $\gamma$ of $E:\Omega(N, \tilde{N})
\longrightarrow\mathbb{R}$ is defined to be the maximal dimension of the space
of variation fields $V$ of variations $f$ that satisfy \eqref{ends in
N and N tilde} and \eqref{pw sm cond} for which $E^{\prime\prime}\left(
0\right)  $ is negative.

Let $\Lambda_{N}$ be as in Example \ref{ex:submanifold_Lagrangian} and set
\begin{align*}
\mathcal{K}  &  \equiv\mathrm{span}\left\{  \left.  J\in\Lambda_{N}\text{
}\right\vert \text{ }J\left(  t\right)  =0\text{ for some }t\in\left(
0,b\right]  \right\}  \text{ and}\\
\mathcal{K}_{b}  &  \equiv\left\{  \left.  J\in\Lambda_{N}\text{ }\right\vert
\text{ }J\left(  b\right)  =0\right\}  .
\end{align*}

Denote by $\Lambda_{N,\tilde{N}}$ the subspace of $\Lambda_{N}$ whose fields
are tangent to $\tilde{N}$ at $\gamma(b)$, and let $A$ be the symmetric
bilinear form
\begin{align}
A  &  :\Lambda_{N,\tilde{N}}\times\Lambda_{N,\tilde{N}}\longrightarrow
\mathbb{R},\nonumber\\
A\left(  J_{1},J_{2}\right)   &  =\left\langle J_{1}^{\prime}\left(  b\right)
-S_{\dot{\gamma}\left(  b\right)  }J_{1},J_{2}\left(  b\right)  \right\rangle
. \label{HK A dfn}%
\end{align}

Thus $A$ is the difference in $T_{\gamma(b)}M$ of the Riccati operator of
$\Lambda$ and the shape operator of $\tilde{N}$ at $\gamma(b)$ in the
direction of $\gamma^{\prime}(1)$.

The following theorem was proven in \cite{HingKa} (cf also \cite{Kal}).

\begin{theorem}
[Hingston-Kalish Theorem]\label{hingst-kalish thmm} The index of $\gamma$ is
equal to
\begin{equation}
\Index A +\text{number of focal pts in }\left(  0,b\right]  -\dim
(\mathcal{K}_{b}\left(  b\right)  \cap T\tilde{N}^{\perp}) \label{our HK}%
\end{equation}
where the focal points are counted with their multiplicities.
\end{theorem}

There is a small difference in the formulation of Theorem
\ref{hingst-kalish thmm} and the statement in \cite{HingKa}. Hingston and
Kalish express the index of $\gamma$ as
\begin{equation}
\Index A +\text{number of focal pts in }(0,b)+m_{T} \label{origin HK}%
\end{equation}
where
\[
m_{T}=\dim\mathrm{Proj}_{\tilde{N}}\mathcal{K}_{b}\left(  b\right)  ,
\]
and $\mathrm{Proj}_{\tilde{N}}$ is orthogonal projection onto $\tilde{N}.$
Since $\dim\mathcal{K}_{b}\left(  b\right)  $ is the multiplicity of $b$ as a
$N$--focal point,
\begin{align*}
m_{T}  &  =\dim\mathcal{K}_{b}\left(  b\right)  -\dim(\mathcal{K}_{b}\left(
b\right)  \cap T\tilde{N}^{\perp})\\
&  =\text{the multiplicity of }b\text{ as a }N\text{-focalpoint}%
-\dim(\mathcal{K}_{b}\left(  b\right)  \cap T\tilde{N}^{\perp}),
\end{align*}
and (\ref{our HK}) and (\ref{origin HK}) agree.

\section{Conjugate Radius and Positive Curvature\label{focal and pos sect}}

\label{sect:proof_Thm_A}

By taking the contrapositive of \eqref{small index implies short} in Lemma
\ref{comp with sing at 0}, we get the following.

\begin{lemma}
\label{long geod lemma} Let $M$ be a complete Riemannian manifold with
$\mathrm{Ric}_{k}\geq k.$ Any geodesic in $M$ of length $>\pi$ has index $\geq
n-k.$
\end{lemma}


To complete the proof of Theorem \ref{Intermeadiate Ricci Thm}, we will need
the following result from Petersen's textbook.

\begin{theorem}
[Theorem 6.5.2 \cite{Pet}]\label{peter morse them} Let $f:N\longrightarrow
\mathbb{R}$ be a smooth proper function defined on a differentiable manifold
$N$. If $b$ is a regular value of $f$ and all critical points in $f$ in
$f^{-1}\left[  a,b\right]  $ have index $\geq m$, then the inclusion%
\[
f^{-1}\left(  -\infty,a\right]  \subset f^{-1}\left(  -\infty,b\right]
\]
is $\left(  m-1\right)  $--connected.
\end{theorem}

The above theorem will be applied to finite dimensional approximations of the
space of closed paths based at some $p\in M$. A standard reference for this is
\cite[Chapter 16]{Miln}.

\begin{proof}
[Proof of Theorem \ref{Intermeadiate Ricci Thm}]Without loss of generality,
assume that $M$ is simply connected. There is nothing to prove if $k=n-1,$ so
we will assume that $k\leq n-2.$

Choose some point $p\in M$ with $\mathrm{conj}_{p}>\frac{\pi}{2},$ and some
number $b$ with $\pi<b<2\,\mathrm{conj}_{p}$. Let $\Omega_{p}$ be the loops
based at $p$ parameterized on $\left[  0,1\right]  .$

If $E:\Omega_{p}\longrightarrow\mathbb{R}$ denotes the energy functional and
$\sigma$ is a critical point of $E$ with $\length\left(  \sigma\right)  >b,$
then by Lemma \ref{long geod lemma},
\begin{equation}
\mathrm{Index}\left(  \sigma\right)  \geq n-k. \label{long geod index}%
\end{equation}

Recall from \cite[Theorem 16.2]{Miln} that there is a finite dimensional
approximation, $N,$ of $\Omega_{p}$ so that for any $a\geq0$, the sublevel set
$N\cap E^{-1}[0,a]$ is a deformation retract of $E^{-1}[0,a]$. Thus we can
apply Theorem \ref{peter morse them} to $N$ and obtain similar statements for
$\Omega_{p}$. We will do that in what follows without any further comment.

Since $\sigma$ has constant speed, it follows from (\ref{Ener vs arc len})
that%
\[
E\left(  \sigma\right)  =\frac{1}{2}\length\left(  \sigma\right)  ^{2}.
\]
Since $n-k\geq2,$ the combination of Theorem \ref{peter morse them} and
Inequality \eqref{long
geod index} gives that the inclusion $E^{-1}\left[  0,b^{2}/2\right]
\hookrightarrow\Omega_{p}$ is $1$--connected, i.e., $\pi_{1}(\Omega_{p}%
,E^{-1}[0,b^{2}/2])=0$. From $\pi_{1}(M)=0$, we get that $\pi_{0}(\Omega
_{p})=0$, and the long exact sequence of the pair $(\Omega_{p},E^{-1}%
[0,b^{2}/2])$ reads%

\[
0=\pi_{1}(\Omega_{p},E^{-1}[0,b^{2}/2])\longrightarrow\pi_{0}(E^{-1}%
[0,b^{2}/2])\longrightarrow\pi_{0}(\Omega_{p})=0.
\]

\medskip

\noindent Thus $\pi_{0}\left(  E^{-1}\left[  0,b^{2}/2\right]  \right)  =0$,
i.e., $E^{-1}\left[  0,b^{2}/2\right]  $ is connected.

Next we show that there is no geodesic loop $\gamma:[0,1]\longrightarrow M$ at
$p$ of length $\ell$ with $\ell\leq b.$ Since $E^{-1}[0,b^{2}/2]$ is
connected, if there were such a loop, then $\gamma$ would be homotopic to a
constant loop through loops $\gamma_{t}$ of energy $\leq b^{2}/2$ and hence of
length
\[
\length\left(  \gamma_{t}\right)  \leq\sqrt{2E\left(  \gamma_{t}\right)  }%
\leq\sqrt{2b^{2}/2}=b
\]
by \eqref{Ener vs arc len}.

On the other hand, by the original Long Homotopy Lemma (see \cite{AbrMey} and
\cite{CheegGrm}), any homotopy of $\gamma$ to a constant loop must pass
through a curve whose length is at least $2\,\mathrm{conj}_{p}>b,$ yielding a
contradiction. Thus there can be no geodesic loop at $p$ with length $\leq b.$
Combining this with (\ref{long geod index}), it follows that all critical
points of $E:\Omega_{p}\longrightarrow\mathbb{R}$ have index $\geq n-k.$ Thus
by Theorem \ref{peter morse them}, $\Omega_{p}$ is $\left(  n-k-1\right)
$--connected, and $M$ is $\left(  n-k\right)  $--connected.
\end{proof}

\section{Quantitative Connectivity Principle}

\label{sect:proof_of_thm_B} In this section we prove Theorem
\ref{quant conn with norm II thm}. Our first result is a simple linear algebra
fact that will allow us to turn the trace estimates of \eqref{tr hyp} and
\eqref{small
future Inequal} into estimates on the values of $S_{t}$.

\begin{proposition}
\label{lin alg prop} Let $A:U\longrightarrow U$ be a self adjoint endomorphism
of an $\ell$--dimensional inner product space. Suppose that there is a
$k\in\left\{  1,2,\ldots,\ell-1\right\}  $ so that for all $k$--dimensional
subspaces $W\subset U,$
\[
\Trace A|_{W} \leq k\cdot\lambda.
\]
Then there is an $\left(  \ell-k+1\right)  $--dimensional subspace $V\subset
U$ so that for all unit $v\in V,$%
\[
\left\langle A v ,v\right\rangle \leq\lambda.
\]

\end{proposition}

\begin{proof}
Let $\left\{  u_{1},\ldots u_{m},v_{1},\ldots,v_{\ell-m}\right\}  $ be an
orthonormal set of eigenvectors for $A$, for which the first $m$-vectors
$\left\{  u_{1},\ldots u_{m}\right\}  $ is the maximal subset that satisfies
\[
\left\langle Au_{i},u_{i}\right\rangle >\lambda.
\]
Then $m\leq k-1.$ Thus
\[
\ell-m\geq\ell-\left(  k-1\right)  ,
\]
and the subspace $V$ spanned by $v_{1},\dots,v_{\ell-k+1}$ satisfies the
conclusion of the proposition.
\end{proof}

Next we state the main lemma for the proof of Theorem
\ref{quant conn with norm II thm}. It will gives us estimates on the index of
a geodesic orthogonal at its endpoints to a submanifold of sufficiently big dimension.

\begin{lemma}
\label{lem:KeyLemma} Let $M$ be a complete Riemannian $n$--manifold with
$\mathrm{Ric}_{k}M\geq k$, and let $N\subset M$ be an $\ell$--dimensional
submanifold for some $\ell\geq k$ that is compact, connected and embedded.

Suppose that $\gamma:\left[  0,b\right]  \longrightarrow M$ is a unit speed
geodesic perpendicular to $N$ at both endpoints such that for all
$k$--dimensional subspaces $W$ tangent to $N$ at the endpoints of $\gamma$, we
have
\begin{equation}
\left\vert \Tr S_{\gamma^{\prime}(0)}|_{W}\right\vert ,\left\vert
\Tr S_{\gamma^{\prime}(b)}|_{W}\right\vert \leq k\cot\left(  \frac{\pi}%
{2}-r\right)  \label{tr hypo ineq}%
\end{equation}
for some $0<r<\frac{\pi}{2}$.
%
Then if the length of $\gamma$ is larger than $2r$, we have
\[
\inx(\gamma)\geq2\ell-n-k+2.
\]

\end{lemma}

\begin{proof}
In what follows, all subspaces of Jacobi fields are understood to be
perpendicular to $\gamma$. Recall that $\Lambda_{N}$ denotes those Jacobi
fields corresponding to variations of $\gamma$ by geodesics that leave $N$
orthogonally at time zero. Let
\[
\mathcal{K}:=\mathrm{span}\left\{  \left.  J\in\Lambda_{N}\left(  0\right)
\text{ }\right\vert \text{ }J\left(  t\right)  =0\text{ for some }t\in\left(
0,b\right]  \,\right\}  .
\]
These are precisely the Jacobi fields that create focal points on $\left(
0,b\right]  ,$ but a given field can vanish multiple times, so we only have
the inequality
\[
\text{number of focal pts in }\left(  0,b\right]  \geq\dim\left(
\mathcal{K}\right)  .
\]

We divide the proof into two cases, depending on the length of $\gamma$:

\medskip

\noindent$\bullet$ \emph{Case 1:} If $b\geq\frac{\pi}{2}+r,$ then by Corollary
\ref{first cor},%
\begin{equation}
\dim\left(  \mathcal{K}\right)  \geq\ell-k+1. \label{big kern}%
\end{equation}
On the other hand, Theorem \ref{hingst-kalish thmm} allows us to estimate the
index of $\gamma$ as%
\begin{equation}
\inx(\gamma)\geq\dim\left(  \mathcal{K}\right)  -\dim(\mathcal{K}_{b}\left(
b\right)  \cap T_{\gamma\left(  b\right)  }N^{\perp}).
\label{eq:lower index bound}%
\end{equation}
Since $\mathcal{K}_{b}\perp\gamma^{\prime}(b)$, we have
\[
\mathcal{K}_{b}\left(  b\right)  \cap T_{\gamma\left(  b\right)  }N^{\perp
}\subset\gamma^{\prime}\left(  b\right)  ^{\perp}\cap T_{\gamma\left(
b\right)  }N^{\perp}.
\]
Together with (\ref{big kern}) and \eqref{eq:lower index bound}, this gives
us
\begin{align*}
\mathrm{index}\left(  \gamma\right)   &  \geq\ell-k+1-\left(  n-1-\ell\right)
\\
&  =2\ell-n-k+2,
\end{align*}
as desired.

\medskip

\noindent$\bullet$ \emph{Case 2: } Suppose $b<\frac{\pi}{2}+r$. If
$\dim\mathcal{K}\geq\ell-k+1$, then we can proceed as in Case 1 to conclude
that
\begin{align*}
\mathrm{Index}\left(  \gamma\right)  \geq\ell-n-k+2.
\end{align*}

\noindent If $\dim\mathcal{K}\leq\ell-k$, set
\[
U_{0}:=\mathcal{K}\left(  0\right)  ^{\perp}\cap T_{\gamma\left(  0\right)
}N,
\]
and note that, since $N$ is $\ell$-dimensional,%
\[
\dim\,U_{0}\geq\ell-\dim\mathcal{K}\geq k.
\]

Next we apply Lemma \ref{sing soon Ric_k Lemma} to the $k$--dimensional
subspaces $W_{0}$ of $U_{0}.$ To justify this, observe that since $U_{0}%
\perp\mathcal{K}(0)$, the space $\mathcal{V}$ in Lemma
\ref{sing soon Ric_k Lemma} is of full index. Thus hypothesis \eqref{tr hypo
ineq} implies there is
%
a subspace $U_{b}$ of $\mathcal{K}(b)^{\perp}$ with the same dimension as
$U_{0}$ so that for all $k$--dimensional subspaces $W\subset U_{b}$,
\begin{equation}
\Tr S_{b}|_{W}\leq k\cot\left(  \frac{\pi}{2}-r+b\right)  .
\label{fnnn tr inequal}%
\end{equation}


On the other hand, using the formula for $A$ in \eqref{HK A dfn}, for any
$k$-dimensional subspace $W$ in $U_{b}\cap T_{\gamma(b)}N$,
\[
\Tr A|_{W}=\Tr S_{b}|_{W}-\Tr S_{\gamma^{\prime}(b)}|_{W}\leq\Tr S_{b}%
|_{W}+|\Tr S_{\gamma^{\prime}(b)}|_{W}|
\]
Combining this with \eqref{tr hypo ineq} and (\ref{fnnn tr inequal}), for any
such $W$,
\begin{equation}
\Tr A|_{W}\leq k\left(  \cot\left(  \frac{\pi}{2}-r+b\right)  +\cot\left(
\frac{\pi}{2}-r\right)  \right)  . \label{Tr A est}%
\end{equation}
Together with Proposition \ref{lin alg prop}, this implies there is a subspace
$V$ of $U_{b}\cap T_{\gamma\left(  b\right)  }N$ with%
\begin{equation}
\dim V\geq\dim(U_{b}\cap T_{\gamma\left(  b\right)  }N)-k+1
\label{large subsp}%
\end{equation}
so that for all unit $v\in V,$
\[
\left\langle Av,v\right\rangle \leq\cot\left(  \frac{\pi}{2}-r+b\right)
+\cot\left(  \frac{\pi}{2}-r\right)  .
\]
Since $2r<b<\frac{\pi}{2}+r$ ,%
\begin{equation}
\left\langle Av,v\right\rangle <0, \label{A neq def}%
\end{equation}
for all $v\in V.$ Thus,
\begin{align}
\Index A  &  \geq\dim V\nonumber\\
&  \geq\dim(U_{b}\cap T_{\gamma\left(  b\right)  }N)-k+1.
\label{index A prelim}%
\end{align}

Observe that \eqref{index A prelim} also holds when $\dim(U_{b}\cap
T_{\gamma(b)}N)<k,$ since in that case \eqref{index A prelim} just reads
$\Index A\geq0$, which is obviously true. Therefore, the rest of the argument
is independent of the dimension of $U_{b}\cap T_{\gamma(b)}N$.

The next goal is to estimate $\dim(U_{b}\cap T_{\gamma(b)}N)$. To this aim,
observe that $U_{b}\perp\mathcal{K}(b)$, thus
\[
U_{b}+T_{\gamma(b)}N\subset\left(  \mathcal{K}(b)\cap T_{\gamma(b)}N^{\perp
}\right)  ^{\perp}.
\]
Since $\mathcal{K}_{b}\subset\mathcal{K}$, we have
\[
\dim(U_{b}+T_{\gamma(b)}N)\leq n-1-\dim(\mathcal{K}_{b}(b)\cap T_{\gamma
(b)}N^{\perp}).
\]


Using also $\dim U_{b} =\dim U_{0} \geq\ell-\dim\mathcal{K} ,$%
\begin{align*}
\mathrm{\dim}( U_{b}\cap T_{\gamma( b) }N)  &  =\dim U_{b} +\dim T_{\gamma( b)
}N -\dim( U_{b}+T_{\gamma\left(  b\right)  }N)\\
&  \geq\ell-\dim\mathcal{K} +\ell-\left(  n-1-\dim\left(  \mathcal{K}%
_{b}\left(  b\right)  \cap T_{\gamma\left(  b\right)  }N^{\perp}\right)
\right) \\
&  =2\ell-n+1-\dim\mathcal{K} +\dim( \mathcal{K}_{b}\left(  b\right)  \cap
T_{\gamma\left(  b\right)  }N^{\perp}) . \label{dim count}%
\end{align*}

Combining this with Theorem \ref{hingst-kalish thmm} and Inequality
(\ref{index A prelim}), we see that
\begin{align*}
\Index\gamma &  \geq\Index A+\dim\mathcal{K}-\dim(\mathcal{K}_{b}\left(
b\right)  \cap T_{\gamma\left(  b\right)  }N^{\perp})\\
&  \geq\dim(U_{b}\cap T_{\gamma\left(  b\right)  }N)-k+1+\dim\mathcal{K}%
-\dim(\mathcal{K}_{b}\left(  b\right)  \cap T_{\gamma\left(  b\right)
}N^{\perp})\\
&  \geq2\ell-n-k+2,
\end{align*}
as claimed.
\end{proof}

\begin{proof}
[Proof of Theorem \ref{quant conn with norm II thm}]The result follows from
Lemma \ref{lem:KeyLemma}, provided we can show that there is no geodesic
$\gamma:\left[  0,1\right]  \longrightarrow M$ with \textrm{len}$\left(
\gamma\right)  =b\leq2r$ and $\gamma(0),\gamma(1)\in N$ with $\gamma^{\prime
}(0),\gamma^{\prime}(1)\perp N.$

Since $N$ is connected and $\pi_{1}(M)=0$, the long exact homotopy sequence of
$\left(  M,N\right)  $ gives $\pi_{1}(M,N)=0$. So without loss of generality
we may assume that $2\ell-n-k+2\geq2.$

Let $\Omega_{N}$ be the space of piecewise smooth paths in $M$ that start and
end in $N$ and are parameterized on $\left[  0,1\right]  .$ Let $E:\Omega
_{N}\longrightarrow\mathbb{R}$ be the energy functional. By Lemma
\ref{lem:KeyLemma}$,$ for all $\tilde{r}>r,$ if $\sigma$ is a critical point
of $E$ with $\sqrt{2E\left(  \sigma\right)  }=\length\left(  \sigma\right)
\geq2\tilde{r},$ then
\begin{align*}
\Index\left(  \sigma\right)   &  \geq2\ell-n-k+2\\
&  \geq2.
\end{align*}
Combining this with Theorem \ref{peter morse them} it follows that $\pi
_{1}\left(  \Omega_{N},E^{-1}\left[  0,2\tilde{r}^{2}\right]  \right)  =0$.
Since $\pi_{1}(M,N)=0$ and $\pi_{0}(\Omega_{N})=0$, the long exact sequence,
\[
\pi_{1}\left(  \Omega_{N},E^{-1}\left[  0,2\tilde{r}^{2}\right]  \right)
\longrightarrow\pi_{0}\left(  E^{-1}\left[  0,2\tilde{r}^{2}\right]  \right)
\longrightarrow\pi_{0}\left(  \Omega_{N}\right)  ,
\]
yields $\pi_{0}\left(  E^{-1}\left[  0,2\tilde{r}^{2}\right]  \right)  =0$.

Now suppose that we have a geodesic $\gamma:[0,1]\longrightarrow M$ with
$\gamma(0),\gamma(1)\in N,$ \textrm{len}$\left(  \gamma\right)  =b\leq2r$, and
$\gamma^{\prime}\left(  0\right)  ,\gamma^{\prime}\left(  1\right)  \perp N.$
Choose $\tilde{r}\in\left(  r,\text{\textrm{focal radius}}\left(  N\right)
\right)  .$ Since $\pi_{0}(E^{-1}[0,2\tilde{r}^{2}])=0$, there is a homotopy
of $\gamma$ to a path in $N$ through paths in $\Omega_{N}$ of energy
$\leq2\tilde{r}^{2}$ and hence of length
\[
\length\left(  \gamma_{t}\right)  \leq\sqrt{2E\left(  \gamma_{t}\right)  }%
\leq\sqrt{2\cdot2\tilde{r}^{2}}=2\tilde{r}.
\]

On the other hand, by the Long Homotopy Lemma (\ref{Lem: long homot}), any
homotopy of $\gamma$ to a path in $N$ must pass through a curve whose length
is $\geq2\foc_{N}>2\tilde{r},$ yielding a contradiction. Thus there can be no
geodesic $\gamma:\left[  0,1\right]  \longrightarrow M$ with $\gamma
(0),\gamma(b)\in N,$ \textrm{len}$\left(  \gamma\right)  =b\leq2r<2\tilde{r}$
and $\gamma^{\prime}(0),\gamma^{\prime}(b)\perp N.$ Thus the inclusion
$N\hookrightarrow M$ is $\left(  2l-n-k+2\right)  $--connected as claimed.
\end{proof}

\section{Quantitative Version of Frankel's Theorem}

\label{sect:Frankel} Frankel's Theorem \cite{Frank1}, a classical result in
Riemannian geometry, asserts that in the presence of positive sectional
curvature, two closed totally geodesic submanifolds whose dimensions add to at
least the dimension of the ambient manifold necessarily intersect. The theorem
has a simple proof using the second variation formula. Since it appeared, the
theorem has been generalized in at least two different directions:

\begin{itemize}
\item The positive sectional curvature has been changed to positive $k$-th
Ricci curvature (see for instance \cite{KenXia} or \cite{Roven}).

\item The totally geodesic restriction on the submanifolds has been changed to
restrictions on their second fundamental forms; in this case, the conclusion
is usually relaxed to an estimate of the relative distance of the submanifolds
(as for instance in \cite{Mor} and \cite{Roven}).

\item Frankel showed in \cite{Frank2} that two minimal surfaces in a manifold
of positive Ricci curvature always intersect. This was generalized to
integrally positive Ricci curvature in \cite{PetWilh}: if the integral of the
portion of the Ricci curvature that is $<n-1$ is small, then the distance
between minimal submanifolds is small.
\end{itemize}

There are more references dealing with Frankel's Theorem, but we include only
the above as an illustration.

Our next result gives the quantitatively optimal generalization of Frankel's
Theorem to intermediate Ricci curvature. In particular, it assumes a weaker
condition on the second fundamental form of the submanifolds than the one
appearing in \cite{Roven}.

\begin{theorem}
\label{quant Frankel} Let $M$ be a complete Riemannian manifold with
$\Ric_{k}M\geq k$. Let $N$ and $\tilde{N}$ be two compact embedded
submanifolds of $M$ so that
\[
\dim N+\dim\tilde{N}\geq\dim M+k-1.
\]
Suppose that for some $r,\tilde{r}\in\left(  0,\frac{\pi}{2}\right)  $ and all
unit vectors $v$ normal to $N$, all $k$--dimensional subspaces $W\subset$
$TN$, all unit vectors $\tilde{v}$ normal to $\tilde{N},$ and all
$k$--dimensional subspaces $\tilde{W}\subset$ $T\tilde{N},$
\begin{equation}
|\Tr S_{v}|_{W}|\leq k\cot\left(  \frac{\pi}{2}-r\right)  \text{ and
}|\Tr S_{\tilde{v}}|_{\tilde{W}}|\leq k\cot\left(  \frac{\pi}{2}-\tilde
{r}\right)  . \label{tr est frankle}%
\end{equation}
Then
\begin{equation}
\mathrm{dist}(N,\tilde{N})\leq r+\tilde{r}. \label{Frank dist}%
\end{equation}

\end{theorem}

\begin{proof}
If $N$ and $\tilde{N}$ intersect, then Inequality (\ref{Frank dist}) is
obvious, so assume that $N\cap\tilde{N}=\varnothing$. It follows that there is
a minimal geodesic segment $\gamma:\left[  0,b\right]  \longrightarrow M$
between $N$ and $\tilde{N}.$

Observe first that the length of $\gamma$ cannot exceed $\frac{\pi}{2}+r$;
otherwise, Inequality (\ref{tr est frankle}) and Corollary \ref{first cor}
give a field in $\Lambda_{N}$ that vanishes before time $b$. This contradicts
the minimality of $\gamma$, thus $b\leq\frac{\pi}{2}+r.$

As before we set
\begin{align*}
\mathcal{K}  &  \equiv\mathrm{span}\left\{  \left.  J\in\Lambda_{N}\text{
}\right\vert \text{ }J\left(  t\right)  =0\text{ for some }t\in\left(
0,b\right]  \right\}  ,\text{ and }\\
\mathcal{K}_{b}  &  \equiv\left\{  \left.  J\in\Lambda_{N}\text{ }\right\vert
\text{ }J\left(  b\right)  =0\text{ }\right\}  .
\end{align*}
Since $\gamma$ is minimal, no field in $\mathcal{K}$ vanishes before time $b,$
in particular
\begin{equation}
\mathcal{K=K}_{b}. \label{K is K b inequal}%
\end{equation}

Further, by Theorem \ref{hingst-kalish thmm},%
\begin{align*}
0  &  =\Index \gamma\\
&  =\Index A +\text{number of focal pts in }\left(  0,b\right]  -\dim(
\mathcal{K}_{b}(b) \cap T\tilde{N}^{\perp})\\
&  \geq\Index A +\dim\mathcal{K} -\dim( \mathcal{K}_{b}\left(  b\right)  \cap
T\tilde{N}^{\perp}) .
\end{align*}

But $\dim\mathcal{K} -\dim( \mathcal{K}_{b}\left(  b\right)  \cap T\tilde
{N}^{\perp}) \geq0,$ so%
\begin{equation}
\Index A =0 \text{ and } \dim\mathcal{K} =\dim( \mathcal{K}_{b}\left(
b\right)  \cap T\tilde{N}^{\perp}) . \label{HK consqs}%
\end{equation}

Combined with $\mathcal{K=K}_{b},$ we have%
\begin{equation}
\mathcal{K}_{b}\left(  b\right)  \subset T\tilde{N}^{\perp}. \label{K in norm}%
\end{equation}

To complete the proof, we show that $\Index A=0$ implies that $b\leq
r+\tilde{r}.$ To do this, use Inequality (\ref{tr est frankle}) and apply
Lemma \ref{sing soon Ric_k Lemma} with $W_{0}\equiv\mathcal{K}(0)^{\perp}\cap
T_{\gamma\left(  0\right)  }N.$ It follows that there is a $\left(  \dim
W_{0}\right)  $--dimensional subspace $H_{b}$ of $T_{\gamma\left(  b\right)
}M$ that is perpendicular to $\gamma^{\prime}\left(  b\right)  $ and
$\mathcal{K}\left(  b\right)  $ so that for all $k$--dimensional subspaces
$W\subset H_{b},$
\begin{equation}
\Tr\left(  S|_{W}\right)  \leq k\cot\left(  \frac{\pi}{2}-r+b\right)  .
\label{final tr}%
\end{equation}

Suppose for the moment that one of these $k$--dimensional subspaces is also
contained in $T_{\gamma\left(  b\right)  }\tilde{N}.$ Then Inequalities
(\ref{tr est frankle}) and (\ref{final tr}) together with the fact that
$\tilde{r}\in\left(  0,\frac{\pi}{2}\right)  $ give%
\begin{align*}
k\left(  \cot\left(  \frac{\pi}{2}-r+b\right)  +\cot\left(  \frac{\pi}%
{2}-\tilde{r}\right)  \right)   &  \geq\Tr\left(  S|_{W}-S_{\dot{\gamma}}%
|_{W}\right) \\
&  =\Tr\left(  A|_{W}\right)  .
\end{align*}

Since $\Index A=0$,
\begin{equation}
\cot\left(  \frac{\pi}{2}-r+b\right)  \geq-\cot\left(  \frac{\pi}{2}-\tilde
{r}\right)  . \label{cot inequal}%
\end{equation}

Since $b-r\leq\frac{\pi}{2},$ $r\in\left(  0,\frac{\pi}{2}\right)  ,$ and
$b>0,$
\[
0<\frac{\pi}{2}-r+b\leq\pi.
\]

Combining this with $\tilde{r}\in\left[  0,\frac{\pi}{2}\right)  \ $and
(\ref{cot inequal}), we can conclude that%
\[
\frac{\pi}{2}-r+b\leq\frac{\pi}{2}+\tilde{r}.
\]
This will give us
\begin{align*}
\mathrm{dist}(N,\tilde{N})  &  =b\\
&  \leq r+\tilde{r},
\end{align*}
once we prove there is a $k$--dimensional subspace $U\subset H_{b}\cap
T_{\gamma\left(  b\right)  }\tilde{N}.$

To do this we set
\[
m=\dim M,\quad n=\dim N,\quad\text{and}\quad\tilde{n}=\dim\tilde{N},
\]
and note that
\begin{align*}
\dim H_{b}  &  =\dim W_{0}\\
&  =\dim(T_{\gamma\left(  0\right)  }N\cap\mathcal{K}(0)^{\perp})\\
&  \geq n-\dim\mathcal{K}.
\end{align*}
Since both $H_{b}$ and $T_{\gamma\left(  b\right)  }\tilde{N}$ are
perpendicular to $\mathcal{K}_{b}\left(  b\right)  $ and $\gamma^{\prime
}\left(  b\right)  ,$
\begin{align*}
\dim(H_{b}+T_{\gamma\left(  b\right)  }\tilde{N})  &  \leq m-1-\dim
\mathcal{K}_{b}\left(  b\right) \\
&  =m-1-\dim\mathcal{K},\text{ by (\ref{K is K b inequal}).}%
\end{align*}
The previous two inequalities give us%
\begin{align*}
\dim(H_{b}\cap T_{\gamma\left(  b\right)  }\tilde{N})  &  =\dim H_{b}+\dim
T_{\gamma\left(  b\right)  }\tilde{N}-\dim(H_{b}+T_{\gamma\left(  b\right)
}\tilde{N})\\
&  \geq n-\dim\left(  \mathcal{K}\right)  +\tilde{n}-\left(  m-1-\dim\left(
\mathcal{K}\right)  \right) \\
&  \geq n+\tilde{n}-m+1\\
&  \geq k,
\end{align*}
as desired.
\end{proof}

\section{Examples}

\label{Sect: Examples} Our first example shows that the Ricci curvature
version of Theorem \ref{Intermeadiate Ricci Thm} is optimal in the sense that
the conclusion cannot be strengthened to say that $M$ is $2$--connected.

\begin{example}
Let $S_{3}^{2}$ be the $2$--sphere with constant curvature $3.$ Then%
\[
\mathrm{Ric}\left(  S_{3}^{2}\times S_{3}^{2}\right)  \equiv3\equiv
\mathrm{Ric}\left(  \mathbb{S}^{4}\right)  ,
\]
and for all $p\in S_{3}^{2}\times S_{3}^{2},$%
\[
\mathrm{conj}_{p}=\mathrm{inj}_{p}=\pi\sqrt{\frac{1}{3}}>\frac{\pi}{2}.
\]

As predicted by Theorem \ref{Intermeadiate Ricci Thm}, $S_{3}^{2}\times
S_{3}^{2}$ is $1$--connected, and since $S_{3}^{2}\times S_{3}^{2}$ is not
$2$--connected, the conclusion of the Ricci curvature version of Theorem
\ref{Intermeadiate Ricci Thm} cannot be strengthened.
\end{example}

The Octonionic projective plane provides an example of Theorem
\ref{Intermeadiate Ricci Thm} that is truly about intermediate Ricci
curvature, as well as an example showing that the conclusion about
connectivity cannot be strengthened.

\begin{example}
$\mathbb{O}P^{2}$ with the canonical metric whose curvatures are in $\left[
1,4\right]  $ is a $16$--manifold with $Ric_{9}\mathbb{O}P^{2}\geq12$ in which
each point has injectivity radius $\frac{\pi}{2}.$ If we multiply this metric
by $\frac{12}{9}$ we get
\[
Ric_{9}\left(  \sqrt{\frac{12}{9}}\mathbb{O}P^{2}\right)  \geq9
\]
and
\[
\mathrm{inj}\left(  \sqrt{\frac{12}{9}}\mathbb{O}P^{2}\right)  =\frac{\pi}%
{2}\frac{\sqrt{12}}{\sqrt{9}}=\frac{\sqrt{3}}{3}\pi>\frac{\pi}{2}.
\]

As predicted by Theorem \ref{Intermeadiate Ricci Thm}, $\mathbb{O}P^{2}$ is
$7$--connected; however, as it is not $8$--connected, we see that the
conclusion of Theorem \ref{Intermeadiate Ricci Thm} is optimal for
$16$--manifolds with $Ric_{9}\geq9$.
\end{example}

The next example shows that Theorem \ref{quant conn with norm II thm} is
optimal in the sense that the estimate (\ref{II hypo}) cannot be replaced with
a weak inequality.

\begin{example}
\label{Ex: cliff torus}View $\mathbb{S}^{3}$ as the unit sphere in
$\mathbb{C}\oplus\mathbb{C}.$ The Clifford torus%
\[
N=\left\{  \left.  \left(  z_{1},z_{2}\right)  \in\mathbb{C}^{2}\text{
}\right\vert \text{ }\left\vert z_{1}\right\vert =\left\vert z_{2}\right\vert
=\frac{1}{\sqrt{2}}\right\}
\]
has
\[
\left\vert \mathrm{II}_{N}\right\vert =1=\cot\left(  \frac{\pi}{2}-\frac{\pi
}{4}\right)  \text{ and }\foc_{N}=\frac{\pi}{4}.
\]

As predicted by Theorem \ref{quant conn with norm II thm}, the inclusion
$N\hookrightarrow\mathbb{S}^{3}$ is $1$--connected, but as it is not
$2$--connected, Inequality (\ref{II hypo}) cannot be be weakened to
$\foc_{N}\geq r.$
\end{example}

Our final example shows that Theorem \ref{quant Frankel} is optimal in the
sense that the estimate (\ref{Frank dist}) cannot be replaced with a strict inequality.

\begin{example}
As before take
\[
N=\left\{  \left.  \left(  z_{1},z_{2}\right)  \in\mathbb{C}^{2}\text{
}\right\vert \text{ }\left\vert z_{1}\right\vert =\left\vert z_{2}\right\vert
=\frac{1}{\sqrt{2}}\right\}  ,
\]
to be the Clifford Torus. Take
\[
\tilde{N}\equiv\left\{  \left.  \left(  z_{1},0\right)  \in\mathbb{C}%
^{2}\text{ }\right\vert \text{ }\left\vert z_{1}\right\vert =1\right\}
\]
to be a coordinate circle. Then
\[
\left\vert \mathrm{II}_{N}\right\vert =1=\cot\left(  \frac{\pi}{2}-\frac{\pi
}{4}\right)  \text{, }\left\vert \mathrm{II}_{\tilde{N}}\right\vert
=0=\cot\left(  \frac{\pi}{2}-0\right)  ,\text{ and }\dist(N,\tilde{N}%
)\leq\frac{\pi}{4}+0,
\]
as predicted by Theorem \ref{quant Frankel}; however, as $\dist(N,\tilde
{N})=\pi/4$, we see that \eqref{Frank dist} cannot be replaced with a strict inequality.
\end{example}

\bigskip


\section{Appendix: Proof of the Long Homotopy Lemma}

In this appendix, we will provide a proof of the Long Homotopy Lemma for
Submanifolds, as stated in Lemma \ref{Lem: long homot} in the introduction.
The proof has some points in contact with \cite[Lemma 4.3]{AbrMey-MSRI}. To
facilitate the reading, we will recall some of the notation used in Lemma
\ref{Lem: long homot}.

$N$ is a closed, embedded submanifold in a Riemannian manifold $M$, and
$\gamma:[0,b]\rightarrow M$ is a unit speed geodesic orthogonal to $N$ at its
endpoints $\gamma(0)$, $\gamma(b)\in N$. We denote by $\nu(N)$ the normal
bundle of $N$ in $M$, and by $N_{0}$ the zero section of $\nu(N)$; finally,
let
\[
\exp_{N}^{\perp}:\nu(N)\longrightarrow M
\]
be the normal exponential map.

Let $D(N,\foc_{N})\subset\nu(N)$ be the set of vectors normal to $N$ with
length less than $\foc_{N}$. Since $\exp_{N}^{\perp}:D(N,\foc_{N}%
)\longrightarrow M$ is a local diffeomorphism, $g^{\ast}:={\exp_{N}^{\perp}%
}^{\ast}g$ is a Riemannian metric on $D(N,\foc_{N})$. In this metric,
$D(N,\foc_{N})$ agrees with the $\foc_{N}$-neighborhood of $N_{0}$, the zero
section of $\nu(N)$. We will denote such $g^{\ast}$-neighborhoods by
$B(N_{0},r)$ for $r\geq0$.

\begin{lemma}
\label{lem:lifts of paths} Let $\alpha:[0,b]\rightarrow M$ be a curve with
$\alpha(0)\in N$ and $\length(\alpha)=\ell<\foc_{N}$. Given a point $\bar
{q}\in N_{0}$ with $\pi(\bar{q})=\alpha(0)$, there is a unique lift
$\bar{\alpha}:[0,b]\rightarrow\nu(N)$ with $\bar{\alpha}(0)=\bar{q}$;
moreover, $\bar{\alpha}$ is entirely contained in $\overline{B(N_{0},\ell)}$.
\end{lemma}

\begin{proof}
There is an open cover $\{ U_{\alpha}\}_{\alpha}$ of $\overline{B\left(
N_{0},\ell\right)  }$ so that each map
\[
\left.  \exp_{N}^{\perp}\right\vert _{U_{\alpha}}:( U_{\alpha},g^{\ast})
\longrightarrow( \exp_{N}^{\perp}(U_{\alpha}),g)
\]
is an isometry for all $\alpha$. The argument then follows the standard proof
for lifts of curves in covering spaces.
\end{proof}

The next lemma considers homotopies
\[
H:[0,b]\times\lbrack0,1]\rightarrow M,\quad H=H(t,s),
\]
sending the three sides
\[
\{0\}\times\lbrack0,1],\quad\lbrack0,b]\times\{0\},\quad\{b\}\times
\lbrack0,1]
\]
of the rectangle $[0,b]\times\lbrack0,1]$ into $N$. To simplify the writing,
we will use $\sqcup$ to denote the union of the above three sides. We will
also use $H_{s}$ to denote the maps
\[
H_{s}:[0,b]\rightarrow M,\quad t\rightarrow H(t,s)
\]
for $s\in\lbrack0,1]$. Finally, define the function $h(s)=\length(H_{s})/2$.
If $H$ has enough regularity (e.g, smooth), $h:[0,1]\rightarrow\mathbb{R}$ is continuous.

\begin{lemma}
\label{short homotopy lemma} Let $H:[0,b]\times\lbrack0,1]\rightarrow M$ be a
homotopy such that

\begin{itemize}
\item $H(\sqcup)\subset N$ and

\item $\length (H_{s})< 2\foc_{N}$ for every $0\leq s\leq1$.
\end{itemize}

Then $H$ has a lift
\[
\tilde{H}:[0,b] \times[ 0,1] \longrightarrow B(N_{0},\foc_{N}) \subset\nu( N)
\]
with $\tilde{H}(\sqcup)$ contained in the zero section $N_{0}$.
\end{lemma}


\begin{proof}
We can assume, without loss of generality, that $H$ is smooth. Thus, as
observed above, $h(s)$ is a continuous function. The condition on the length
of the curves $H_{s}$ implies that there is a number $r$ with $0<r<\foc_{N}$
so that the length of any $H_{s}$ does not exceed $2r$.

Consider the curves
\[
L_{s}:[0,h(s)]\rightarrow M,\qquad L_{s}(t)=H(t,s),
\]
and
\[
R_{s}:[h(s),b]\rightarrow M,\qquad R_{s}(t)=H(t,s).
\]
By the choice of $h(s)$, the lengths of $L_{s}$ and $R_{s}$ do not exceed $r$.
Lemma \ref{lem:lifts of paths} gives us lifts
\[
\bar{L}_{s}:[0,h(s)]\rightarrow\nu(S),\qquad\bar{R}_{s}:[h(s),b]\rightarrow
\nu(S)
\]
sending $\sqcup$ continuously into $N_{0}$. Moreover, the uniqueness of the
lifts together with their being constructed using local inverses of $\exp
_{N}^{\perp}|_{U_{\alpha}}$ shows that $\bar{L}(t,s):=\bar{L}_{s}(t)$ and
$\bar{R}(t,s):=\bar{R}_{s}(t)$ are continuous in their respective domains.

We claim that the map $\bar{H}(t,s)$ defined by
\[
\bar{H}(t,s):=%
\begin{cases}
\bar{L}(t,s),\quad\text{ for }0\leq t\leq h(s),\\
\bar{R}(t,s),\quad\text{ for }h(s)\leq t\leq1,
\end{cases}
\]
obtained by pasting together $\bar{L}$ and $\bar{R}$ is continuous. By the
Pasting Lemma we just need to show that $\bar{L}(h(s),s)=\bar{R}(h(s),s)$ for
all $s\in\lbrack0,1]$. This can be done using a clopen argument; define
\[
A:=\{s\in\lbrack0,1]\,:\,\bar{L}(h(s^{\prime}),s^{\prime})=\bar{R}%
(h(s^{\prime}),s^{\prime})\text{ for all }s^{\prime}\leq s\,\},
\]
and observe that $0\in A,$ and $A$ is clearly closed. It remains to see that
it is also open. This is a consequence of the way lifts are constructed.
Indeed, suppose $s_{0}\in A$, let $p_{0}=H(h(s_{0}),s_{0})$, and choose some
$U_{\alpha}$ containing $\bar{L}(h(s_{0}),s_{0})=\bar{R}(h(s_{0}),s_{0})$
where $\exp_{N}^{\perp}$ is a diffeomorphism. Then $p_{0}$ lies in $\exp
_{N}^{\perp}(U_{\alpha})$, so
\[
\bar{L}(h(s),s)=\exp_{N}^{\perp}|_{U_{\alpha}}^{-1}\circ H(h(s),s),\hspace
{0.14in}\text{and}\hspace{0.14in}\bar{R}(h(s),s)=\exp_{N}^{\perp}|_{U_{\alpha
}}^{-1}\circ H(h(s),s),
\]
for $s$ close to $s_{0}$. Thus $A$ is open, and by connectedness, $A=[0,1]$.
\end{proof}


\begin{lemma}
\label{no leift lemma} Let $\gamma:[0,b]\rightarrow M$ be a unit speed
geodesic with $\gamma(0)$, $\gamma(b)\in N$, $b<2\foc_{N}$, and $\gamma
^{\prime}(0)$, $\gamma^{\prime}(b)\perp N$. Then there is no lift of $\gamma$
to a curve $\hat{\gamma}:[0,b]\longrightarrow B(N_{0},\foc_{N})$ with
$\hat{\gamma(0)}$, $\hat{\gamma(b)}\in N_{0}$.
\end{lemma}

\begin{proof}
Let $\tilde{\gamma}$ be a lift of $\gamma$ to $\nu(N)$ with $\tilde{\gamma
}(0)
\in N_{0}$. Since $\exp^{\perp}$ is a local diffeomorphism in $B(N_{0}%
,\foc_{N})$, we have
\[
\tilde{\gamma}(t)=t\gamma^{\prime}(0)
\]
in any interval $[0,c]$ with $b/2<c<\foc_{N}$.
%
Thus,
\[
\left\vert \tilde{\gamma}\left(  c\right)  \right\vert =\dist\left(
\tilde{\gamma}\left(  c\right)  ,N_{0}\right)  =c>\frac{b}{2}.
\]
Since $\dist(\tilde{\gamma}(c),\tilde{\gamma}(b))<b-c<b/2,$ it follows that
$\tilde{\gamma}\left(  b\right)  $ cannot lie in $N_{0}.$

\end{proof}

\begin{proof}
[Proof of Lemma \ref{Lem: long homot} (The long homotopy Lemma)]By Lemma
\ref{no leift lemma}, $H$ has no lift to a homotopy with values in $B\left(
N_{0},\foc_{N}\right)  $ whose curve end points are in $N_{0}.$ So by Lemma
\ref{short homotopy lemma}, for some $s\in\left[  0,1\right]  $ the curve
\[
t\longmapsto H_{s}\left(  t\right)
\]
must be longer than $2\foc_{N}.$
\end{proof}

\end{document}